\documentclass[12pt]{article}
\usepackage{amsmath, amsfonts, amssymb, latexsym, wasysym, graphicx, mathtools,  wasysym, enumitem, pifont}

\usepackage[all]{xy}

\usepackage{hyperref}
\hypersetup{colorlinks,citecolor=blue,linkcolor=blue}

\title{The Higher Connectivity at Infinity of Mapping Class Groups}
\author{Michael Mihalik}

\newtheorem{theorem}{Theorem}[section]

\newtheorem{lemma}[theorem]{Lemma}

\newtheorem{corollary}[theorem]{Corollary}

\newtheorem{remark}[theorem]{Remark}

\newenvironment{proof}{\addvspace{12pt}\noindent{\bf Proof:}}{
$\Box$\par\addvspace{12pt}}

\numberwithin{equation}{section}

\newcounter{definitionnum}

\date{\today}

\begin{document}

\maketitle

\begin{abstract} 
The higher connectivity at infinity for mapping class groups of surfaces with boundary components and punctures is understood with the exceptions of the mapping class groups for the closed surfaces of genus 3 and 4. In this paper we prove a general simply connected at infinity result for finitely presented groups that implies all mapping class groups of closed surfaces of genus $\geq 3$ are simply connected at infinity. As these groups are duality groups the Proper Hurewicz Theorem implies that they are $(n-2)$-connected at infinity where $n$ is the dimension of the group. Combining this result with earlier work we give a complete list of all mapping class groups and their connectivity at infinity.
\end{abstract}

\maketitle

\section{Introduction}
Simple connectivity at infinity is perhaps the most important asymptotic invariant of finitely presented groups and open manifolds. It is a quasi-isometry invariant for groups \cite{BR93}. For manifolds of dimension $\geq 3$ this invariant completely determines whether or not a contractible open manifold is homeomorphic to Euclidean space and when combined with the Proper Hurewicz Theorem, it can be used to determine the higher connectivity at infinity of certain groups. An open contractible (topological, differentiable or PL) manifold of dimension $\geq 3$ is homeomorphic to Euclidean space if and only if it is simply connected at infinity. This result is a combination of the work of a number of authors. It was first proved by J. Stallings \cite{St62} for PL-manifolds in dimensions $\geq 5$. E. Luft \cite{Luf67} and L. Sibenmann \cite{Sie68} produced topological versions of Stallings Theorem. Once the Poincare Conjectures were proved in dimensions 3 and 4, these results were extended to dimension 3 (see \cite{HP70} and \cite{HP71}) and dimension 4 (see Corollary 1.2 of M. Freedmann \cite{Fre82}). M. Davis \cite{Davis83} used right angled Coxeter groups to produce closed manifolds with infinite fundamental group in dimensions $n\geq 4$, whose universal covers were contractible but not homeomorphic to $\mathbb R^n$. Davis did this by showing these fundamental groups are not simply connected at infinity.

The following is our Main Theorem:

\begin{theorem}\label{Main} Suppose the group $G$ has a finite presentation $\langle S:R\rangle$ and $G$ contains a free abelian subgroup $A$ of rank $\geq 3$ with free generating set $T$ satisfying (1) and (2). Then $G$ is simply connected at infinity.  

\noindent  (1) For $r\in R$ there is $t_r\in T$ such that $t_r$ commutes with each letter of $r$.

\noindent  (2) If $s\in S$ is a letter of $r\in R$ then some $z\in T-\{t_r\}$ commutes with $s$. 
\end{theorem}
First observe that if $\langle S:R\rangle$ is a finite presentation of a group and $s\in S$ does not appear in any relation of $R$, then $G$ splits as a free product $\mathbb Z_s\ast \langle S-\{s\}\rangle$ and $G$ is simply connected at infinity if and only if $\langle S-\{s\}\rangle$ is simply connected at infinity. On the other hand, if $\langle S:R\rangle$ satisfies the hypothesis of Theorem \ref{Main} and for each $s\in S$ there is $r_s\in R$ such that $s$ is one of the letters of $r_s$, then $G$ is one-ended (see \S \ref{PMT}). 

The principal application of this theorem is to mapping class groups. Let $F=F_{g,r}^s$ be an orientable surface of genus $g$ with $s$ punctures and $r$ boundary components. The {\it mapping class group} $\Gamma_{g,r}^s$  is the group of isotopy classes of orientation preserving diffeomorphisms of $F$ which preserve the punctures of $F$ individually and restrict to the identity on $\partial F$. Here the isotopies keep $\partial F$ fixed point-wise. It follows from geometry results in the 1960's (W. Baily \cite{B60} and Deligne-Mumford \cite{DM69}) that $\Gamma_{g,0}^0$ is finitely presented. (See \cite{FM12}, P128).
Also see \cite{McC75} and Theorem 4.3.D of N. Ivanov's manuscript - Mapping Class Groups (December 21, 1998) 

\noindent https://s3.amazonaws.com/nikolaivivanov/Ivanov1999MappingClassGroups.pdf. 

A straightforward corollary of Theorem \ref{Main} is:
\begin{corollary} \label{MCG34}
The mapping class groups for the closed surfaces of genus $g\geq 3$ are simply connected at infinity.
\end{corollary}
These groups are duality groups of dimension $4g-5$. Combining Corollary \ref{MCG34} and the Proper Hurewicz Theorem  ([Theorem 17.1.6, \cite{G}]) implies they are $4g-7$ connected at infinity. For $g\geq 5$, the mapping class groups of closed surfaces are known to satisfy an even stronger condition than simple connectivity at infinity (see  [Theorem 3.3, \cite{ABDDY12}]). The approach used to prove that result has similarities to the proof of our main theorem.
 
 The connectivity at infinity of mapping class groups other than $\Gamma_{3,0}^0$ and $\Gamma_{4,0}^0$ was established in [Theorem 2.4.55, \cite{Man26}]. For $2g+s+t>2$, define $d(g,0,0)=4g-5$; $d(g,r,s)=4g+2r+s-4$ when $g>0$ and $r+s>0$; and $d(0,r,s)=2r+s-3$. Combining with our results produces:
 \begin{theorem}\label{MCGSCINF}
 Let $\Gamma_{g,r}^s$ be the mapping class group of the surface of genus $g$ with $s$ punctures and $r$ boundary components. 
 
1)  $\Gamma_{g,r}^s$ is finite if and only if $g=0$, $r=0$ and $s\leq 3$,  or $g=0$, $r=1$ and $s=0$.

2) $\Gamma_{g,r}^s$ is infinite, virtually free (and hence simply connected at infinity) if and only if $(g,r,s) \in \{ (0,0,4), (0,1,1),  (0,1,2), (0,2,0), (1,0,0), (1,0,1) \}$. 

3) $\Gamma_{g,r}^s$ is virtually an extension of two non-trivial finitely generated free groups  if and only if $(g,r,s) \in \{(0,0,5), (0,1,3), (0,2,1), (1,1,0), (1,0,2) \}$. In this case, $\Gamma_{g,r}^s$ is 1-ended and semistable at infinity, but not simply connected at infinity. 

4) Otherwise, $\Gamma_{g,r}^s$ is $d(g,r,s)-2$ connected at infinity. In particular, all remaining mapping class groups are 1-ended and simply connected at infinity. 
 \end{theorem}
 We establish some notation and give basic definitions of simple connectivity at infinity and van Kampen diagrams for finitely presented groups in \S \ref{SCIVK}. 
 Four principal algebraic lemmas are established in \S \ref{PL}. In \S \ref{OutL} we outline our argument. Beginning with a van Kampen diagram $V$ for a simple closed curve in a Cayley 2-complex for our group, we use the lemmas of \S \ref{PL} to attach 3-cells to the Cayley 2-complex and $V$. When finished $V$, along with its 3-cells will be a 3-ball $B^3$ and we map this 3-ball to our Cayley 3-complex in such a way so that on $\partial B^3-int(V)$ all cells are mapped ``far out" in the space. This establishes simple connectivity at infinity of our space and group. We prove our main theorem in \S \ref{PMT}. Useful presentations of mapping class groups of closed surfaces of genus $\geq 3$ given in \cite{Ge01} are described in \S \ref{PMCG}. We point out how these presentations satisfy the hypotheses of our main theorem so that these groups are simply connected at infinity.
 
 {\it Acknowledgement:} We are grateful for the help provided by Dan Margalit and Denis Osin. 
 
\section{Simple Connectivity at Infinity and van Kampen Diagrams} \label{SCIVK}
 A topological space $X$ is {\it simply connected at infinity} if for each compact set $C\subset X$ there is a compact set $D\subset X$ such that loops in $X-D$ are homotopically trivial in $X-C$. Suppose $X$ is a locally finite and connected CW-complex. If $A$ is a subcomplex of $X$ then define $St(A)$ (the {\it star} of $A$) to be the (full) subcomplex of $X$ with vertex set equal to the vertex set of $A$ along with all vertices that share an edge with a vertex of $A$. A cell of $X$ belongs to $St(A)$ if all of its vertices belong to $A$. So $St(A)$ is analogous to a ball of radius 1 about $A$. Inductively define $St^n(A)=St(St^{n-1}(A)$ for all integers $n\geq 1$. If $v$ and $w$ are vertices of $X$ then $v\in St^n(w)$ if and only if there is an edge path in $X$ between $v$ and $w$ of length $\leq n$. It is straightforward to see that $X$ is simply connected at infinity if and only if for some (any) vertex $\ast\in X$ and any integer $n>0$ there is an integer $M(n)$ such that edge path loops in $X-St^{M(n)}(\ast)$ are homotopically trivial in $X-St^n(\ast)$. We use this as our definition of simply connected at infinity in the remainder of the paper. A finitely presented group $G$ is {\it simply connected at infinity} if for some (equivalently any) connected finite complex $Y$ with $\pi_1(Y)=G$, the universal cover of $Y$ is simply connected at infinity. 
 If $\mathcal P=(S:R)$ is a finite presentation for a group $G$, let $\Gamma(\mathcal P)$ be the Cayley 2-complex for $\mathcal P$. So the vertex set of $\Gamma(\mathcal P)$ is $G$, the 1-skeleton of $\Gamma(\mathcal P)$ is the Cayley graph of $G$ with respect to $S$ (so each edge is directed and labeled by an element of $S$). For each vertex $v\in \Gamma(\mathcal P)$ there is a 2-cell with edge path boundary labeled by $r$ for each $r\in R$. The space $\Gamma(\mathcal P)$ is simply connected and $G$ acts freely and cocompactly (on the left) of $\Gamma(\mathcal P)$ by covering transformations. Hence in order to show that $G$ is simply connected at infinity, it suffices to show that $\Gamma(\mathcal P)$ is simply connected at infinity. In fact, if $\ast$ is the identity vertex of $\Gamma(\mathcal P)$ then $G$ is simply connected at infinity if for any integer $n$ there is an integer $M(n)$ such that any edge path loop in $\Gamma(\mathcal P)-St^{M(n)}(\ast)$ is homotopically trivial in $\Gamma(\mathcal P)-St(\ast)$.
 
 Let $F(S)$ be the free group on the set $S$ and $q:F(X)\to G=\langle S:R\rangle$ be the quotient map taking $s\to s$ for each $s\in S$. The kernel of $q$ is $N(R)$, the normal closure of $R$ in $F(S)$. Each element of $N(R)$ is a product of conjugates of elements in $R$ and hence if $a\in N(R)$ is reduced in $F(S)$, then there is a planar diagram $V(a)$ for $a$ called a {\it van Kampen diagram} for $a$. If $a=b_1r_1b_1^{-1}\cdots b_nr_nb_n^{-1}\in N(R)$ (with $r_i\in R^{\pm1}$) simply draw the obvious diagram for each  $b_ir_ib_i^{-1}$ in the plane, at a base point $\ast$, in circular order around $\ast$. Then ``fold" out consecutive letters of the form $tt^{-1}$. The result is a planar diagram $V(a)$, with boundary labeled by the reduced word for $a$, each edge is labeled by an element of $S^{\pm1}$ and each 2-cell by an element of $R$. Given a vertex $v(\in G)$ of $\Gamma(\mathcal P)$ there is an obvious map of $V(a)$ to $\Gamma(\mathcal P)$ that takes $\ast$ to $v$. 
 
 Suppose $\alpha$ is an edge path loop in $\Gamma(\mathcal P)-St^{M(n)}(\ast)$ and each simple closed sub-loop of $\alpha$ is homotopically trivial in $\Gamma(\mathcal P)-St^n(\ast)$ then $\alpha$ is homotopically trivial in $\Gamma(\mathcal P)-St^n(\ast)$. Hence we only need consider simple closed curves in our proof of simple connectivity at infinity for $G$. If $\alpha$ is a simple closed curve in  $\Gamma(\mathcal P)-St^{M(n)}(\ast)$ then let $w(\alpha)\in N(R)$ be its edge path label. The van Kampen diagram $V=V(w(\alpha))$ is a 2-disk and has boundary a simple edge path loop labeled $w(\alpha)$. The boundary of each 2-cell is labeled by an element of $R$. There is a map $Q:V\to \Gamma(\mathcal P)$ that maps the boundary of $V$ to $\alpha$, each vertex, edge and 2-cell of $V$ to a vertex, edge and 2-cell (respectively) of $\Gamma(\mathcal P)$. It is convenient to label vertices, edges and 2-cells of $V$ by their images under $Q$. We only use $Q(v)(=v)$ to make it clear that we are in $\Gamma(\mathcal P)$.

 \section{Principal Lemmas}\label{PL}

In this section, $G$ is a group with finite generating set $S=\{s_1,\ldots s_n\}$. Let $\Gamma$ be the Cayley graph of $G$ with respect to $S$ and let $\ast$ be the identity vertex of $\Gamma$. If $x\in \{s_1^{\pm 1}, \ldots s_n^{\pm 1}\}$ has infinite order then let $r_x$ be the edge path ray in $\Gamma$, beginning at $\ast$ where each edge is labeled $x$ and let $l_x$ be the edge path line at $\ast$ where each edge is labeled $x$. Note that if $v(\in G)$ is a vertex of $\Gamma$ then $v\cdot r_x$ is the edge path ray at $v$ with all edges labeled $x$.  If $v$ and $w$ are vertices of $\Gamma$ (elements of $G$) then $d(v,w)$ is defined to be the length of the shortest path between them in $\Gamma$. For $g\in G$ we define $|g|=d(g,\ast)$ where $\ast$ is the identity vertex of $\Gamma$. If $A\subset S$, and $v^{-1}w\in \langle A\rangle$ then $d_A(v,w)$ is the length of a shortest edge path between $v$ and $w$, each of whose edges is labeled by an element of $A^{\pm 1}$. Define $|a|_A=d_A(a,\ast)$ for $a\in A$.   
 

Next suppose that $\{x,y,z\}\subset S^{\pm1}$ generates a subgroup of $G$ isomorphic to $\mathbb Z^3$. Define the {\it half plane} $P(x^+,y)$ (respectively $P(x^-,y)$) to be the full subgraph of $\Gamma$ with vertices in the set $\{x^iy^j \ |\ i\geq 0$\} (respectively $i\leq 0$). Observe that if $v$ is a vertex of $\Gamma$ then 
$$v\cdot l_y= (v\cdot P(x^+,y))\cap (v\cdot P(x^-,y))$$ 
and for any $p\in v\cdot P(x^+,y)$ (respectively $q\in v\cdot P(x^-,y)$) there is an edge path from $v\cdot l_y$ to $p$ (respectively to $q$) with each edge label equal to $x$ (respectively $x^{-1}$). 
Define the {\it half space} $R(x^+,y,z)$ (respectively $R(x^-,y,z)$), to be the full subgraph of $\Gamma$ with vertices in the set $\{x^iy^jz^k \ |\ i\geq 0$\} (respectively $i\leq 0$).

\begin{figure}
\vbox to 3in{\vspace {-2in} \hspace {-.7in}
\hspace{-1 in}
\includegraphics[scale=1]{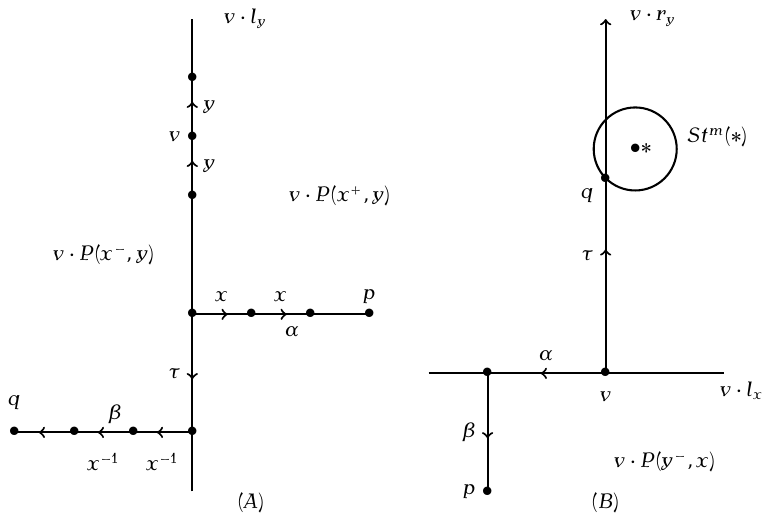}
\vss }
\vspace{.2in}
\caption{Half Planes} 
\label{Fig1}
\vspace{-.1in}
\end{figure}

\begin{lemma}\label{L1}
Given an integer $m$ there is an integer $M_1(m)$ such that if $\{x,y\}\subset S$ such that $\{x,y\}$ generates a subgroup of $G$ isomorphic to $\mathbb Z^2$, and  $v\in G$ such that $(v\cdot l_y)\cap St^{M_1(m)}(\ast)=\emptyset$, then at most one of the half planes $v\cdot P(x^+,y)$ or $v\cdot P(x^-,y)$ intersects $St^m(\ast)$. 
\end{lemma}
\begin{proof} 
There are only finitely many pairs $x,y\in S$ such that the subgroup $\langle x,y\rangle$ is isomorphic to $\mathbb Z^2$. Hence it suffices to prove the lemma for one such pair. There are only finitely many elements  $a\in \langle x,y\rangle$ such that $|a|\leq 2m$. Choose $L_1(m)$ such that $|a|_{\{x,y\}}< L_1(m)$ for each of these elements. Define $M_1(m)=L_1(m)+m$ and suppose $(v\cdot l_y)\cap St^{M_1(m)}(\ast)=\emptyset$. (See Figure \ref{Fig1}(A).) If $p\in (v\cdot P(x^+,y))\cap St^m(\ast)$ then there is an edge path $\alpha$ in $v\cdot P(x^+,y)$ from $v\cdot l_y$ to $p$ with each edge labeled $x$. Note that $|\alpha |\geq L_1(m)$ (since $(v\cdot l_y)\cap St^{L_1(m)}(St^m(\ast))=\emptyset $ and $p\in St^m(\ast)$).  If $q\in (v\cdot P(x^-,y))\cap St^m(\ast)$ then there is an edge path $\beta$ in $v\cdot P(x^-,y)$ from $v\cdot l_y$ to $q$ with each edge labeled $x^{-1}$. Note that $|\beta |\geq L_1(m)$. If $\tau$ is a shortest edge path in $v\cdot l_y$ from the initial point of $\alpha$ to the initial point of $\beta$ then the edge path $(\alpha^{-1}, \tau, \beta)$  from $p$ to $q$ is geodesic in the letters $x,y$ and has length $| \alpha |+|\tau|+|\beta|>2L_1(m)+|\tau | (>L_1(m))$. This is impossible since $d(p,q)\leq 2m$ and $d_{\{x,y\}}(p,q)>L_1(m)$ . 
\end{proof}

\begin{lemma}\label{L2}
Suppose $\{x,y\}\subset S$ such that $\{x,y\}$ generates a subgroup of $G$ isomorphic to $\mathbb Z^2$. If  $v\in \Gamma-St^{M_1(m)}(\ast)$ and 
$(v\cdot r_y)\cap St^m(\ast)\ne\emptyset$, then  $(v\cdot P(y^-,x))\cap St^m(\ast)=\emptyset$. In particular, $(vy^n\cdot l_x)\cap St^m(\ast)=\emptyset$ for $n\leq 0$ and $v\cdot r_{y^{-1}}\cap St^m(\ast)=\emptyset$. 
\end{lemma}
\begin{proof} Recall that $M_1(m)=L_1(m)+m$ and if $d(p,q)\leq 2m$ and $q\in p\cdot \langle x,y\rangle$ then $d_{\{x,y\}}(p,q)< L_1(m)$. Let $q$ be a vertex of $(v\cdot r_y)\cap St^m(\ast)$. Then the edge path $\tau$ from $v$ to $q$ along $v\cdot r_y$ has each edge labeled $y$ and is of length $>L_1(m)$. (See Figure \ref{Fig1}(B).) Say $p\in (v\cdot P(y^-,x))\cap St^m(\ast)$. An $(x,y)$-geodesic from $v$ to $p$ has the form $(\alpha, \beta)$ where each edge of $\alpha$ is labeled by $x$ or each edge is labeled by $x^{-1}$ and each edge of $\beta$ is labeled by $y^{-1}$. The edge path $(\tau^{-1}, \alpha, \beta)$ is an $(x,y)$-geodesic from $q$ to $p$. But the length of this $(x,y)$-geodesic is $\geq L_1(m)$. Since $p,q\in St^m(\ast)$, $d(p,q)\leq 2m$. This contradicts the definition of $L_1(m)$. 
\end{proof}

\begin{lemma}\label{L3}
Suppose $\{x,y,z\}\subset S^{\pm1}$ such that $\{x,y,z\}$ generates a subgroup of $G$ isomorphic to $\mathbb Z^3$. Given an integer $m>0$ there is an integer $M_2(m)>M_1(m)$ such that if  $v\in \Gamma-St^{M_2(m)}(\ast)$ and 
$(v\cdot r_x)\cap St^m(\ast)\ne\emptyset$, then  
$$(v\cdot R(x^-,y,z))\cap St^m(\ast)=\emptyset.$$ 
\end{lemma}
\begin{proof} There are only finitely many triples $x,y,z\in S$ such that the subgroup $\langle x,y,z\rangle$ is isomorphic to $\mathbb Z^3$. Hence it suffices to prove the lemma for one such triple. There are only finitely many elements  $a\in \langle x,y,z\rangle$ such that $|a|\leq 2m$. Choose $L_2(m)$ such that $|a|_{\{x,y,z\}}< L_2(m)$ for each of these elements. Define $M_2(m)=L_2(m)+m$. If $d(p,q)\leq 2m$ and $q\in p\cdot \langle x,y,z\rangle$ then $d_{\{x,y,z\}}(p,q)< L_2(m)$. Let $q$ be a vertex of $(v\cdot r_x)\cap St^m(\ast)$. Then the edge path $\tau$ from $v$ to $q$ along $v\cdot r_x$ has each edge labeled $x$ and is of length $>L_2(m)$. (See Figure \ref{Fig1}(B) with $v\cdot r_y$ replaced by $v\cdot r_x$.) Say $p\in (v\cdot R(x^-,y,z))\cap St^m(\ast)$. An $(x,y,z)$-geodesic from $v$ to $p$ has the form $(\alpha, \beta,\gamma)$ where each edge of $\alpha$ is labeled by $x^{-1}$ and each edge of $\beta$ is labeled by $y$ or each by $y^{-1}$ and each edge of $\gamma$ is labeled by $z$ or each by $z^{-1}$. The edge path $(\tau^{-1}, \alpha, \beta,\gamma)$ is an $(x,y,z)$-geodesic from $q$ to $p$. But the length of this $(x,y,z)$-geodesic is $\geq L_2(m)$. Since $p,q\in St^m(\ast)$, $d(p,q)\leq 2m$. This contradicts the definition of $L_2(m)$. 
\end{proof}

\begin{lemma}\label{L5}
If $v\in \Gamma$ and $m>0$ then there exists an integer $T(v,m)>0$ such that for all $t\geq T(v,m)$ and any pair $\{x,y\}\subset S^{\pm 1}$ such that $\{x,y\}$ generates a subgroup of $G$ isomorphic to $\mathbb Z^2$, we have:
$(v\cdot r_x(t)\cdot P(x^+,y))\cap St^m(\ast)=\emptyset$. (Note that $r_x(t)=x^t$.)
\end{lemma}

\begin{proof}
There are only finitely many such pairs $x,y$ so it suffices to prove the lemma for one such pair. For distinct positive integers  $a,b$, the lines $v\cdot r_x(a)\cdot  l_y$ and $v\cdot r_x(b)\cdot l(y)$ are disjoint. Hence only finitely many of these lines can intersect $St^m(\ast)$. Choose $T(v,x,m)$ large enough so that for all $t\geq T(v,x,m)$, $(v\cdot r_x(t)\cdot l(y))\cap St^m(\ast)=\emptyset$. 
\end{proof}

\section{Outline of the Proof of the Main Theorem}\label{OutL}
Suppose $\mathcal P_1=\langle S,R_1\rangle$ is a finite presentation of $G$, the subgroup $\langle t_1,\ldots, t_p\rangle$ of $G$ is isomorphic to $\mathbb Z^p$ and for each $r\in R_1$ some $t_i$ commutes with every letter of $r$. Without loss, we may assume that if $r\in R_1$ then no subword (cyclically) of $r$ is trivial in $G$. Let $T=\{t_1,\ldots, t_p\}$. Let $\mathcal P_2=\langle S\cup T, R_1\cup R_2\rangle$ be a Tietze expansion of $\mathcal P_1$ (use Tietze moves to add each $t_i$ as a generator). Also assume that if two generators commute, then that commutation relation is in $R_1\cup R_2$. 
First we describe the space we work in and then we give the idea of the proof. Let $X^2$ be the Cayley 2-complex for $\mathcal P_2$. So the 1-skeleton of $X^2$ is the Cayley graph of $G$ with respect to $S\cup T$, each 2-cell has boundary labeled by an element of $R_1\cup R_2$, $\pi_1(X^2)=1$ and $G$ acts freely and co-compactly on $X^2$. We will add 3-cells to $X^2$.  Suppose $r\in R_1$ and $t\in T$ where each letter of $r$ commutes with $t$. If $C$ is a 2-cell of $X^2$ with boundary labeled by $r$ then we attach the obvious 3-cell with base labeled by $C$, top labeled by $C$ and sides labeled by conjugation relations involving $t$ and the letters of $r$. If three letters of $S\cup T$ commute, then add the corresponding 3-cells to each vertex of $X^2$. Call the resulting $3$-complex $X$. The group $G$ acts freely and co-compactly on the simply connected space $X$. Observe that if $C$ is a 2-cell of $X^2$ with $\partial C$ labeled by an element of $R_1$, and $x\in T$ commutes with each label in  $\partial C$ then there is a homeomorphism (onto its image) 
$$H_C:C\times (-\infty, \infty)\to X \hbox{ where }$$ 

\noindent 1) $H_C(c,0)=c$ for all $c\in C$, 

\noindent 2) for each integer $n$, $H_C(C\times [n,n+1])$ is one of our three cells and 

\noindent 3) $x^kH_C(c,n)=H_C(c,n+k)$ for all integers $n$ and $k$. 

In particular, if $n$ is an integer, then $H_C(C\times [n,n+1])$ is the translate of the 3-cell corresponding to $C$ and $x$,  by $x^n$.  Finally:

\medskip

\noindent 4) If $v$ is a vertex of $C$ then $H_C(\{v\}\times (-\infty,\infty)$ is  $v\cdot l_x$. 

\medskip

Now we outline the proof. Choose $W$ such that each element of $R_1$ has length $\leq W$. 
\begin{remark}\label{R1}
We assume that the constants $M_1(m)$, $M_2(m)$, $L_1(m)$, $L_2(m)$ and $T(v,m)$ are defined for our presentation $\mathcal P_2$. We also assume that $M_1(m)>m+W$ and $M_2(m)\geq M_1(m)+W$. 
\end{remark}
 
Any edge path loop is homotopic by a ``small" homotopy to an edge path loop with all labels in $S$. Hence it suffice to show that any edge path loop $\alpha$ in $X-St^{M_2(M_1(m)+W)}(\ast)$ with each edge label in $S$, is homotopically trivial in $X-St^m(\ast)$. It suffices to show that each sub-loop of $\alpha$ that does not intersect itself is homotopically trivial in $X-St^m(\ast)$. Hence we reduce to the case that $\alpha$ is embedded and all letters of $\alpha$ are in $S$.  Let $V$ be a van Kampen diagram for $\alpha$ where each 2-cell is labeled by an element of $R_1$, and let $Q:V\to X$ be the natural map taking the boundary of $V$ to $\alpha$. Since $\alpha$ is embedded, $V$ is homeomorphic to a 2-disk. It is convenient for our figures to label 2-cells, edges and vertices of $V$ by the label of their image under $Q$. We will use $V$ to build a 3-ball $B^3$ with $V$ embedded in $S^2$, the (2-sphere) boundary of $B^3$. We will also construct a map of $\bar Q:B^3\to X$ an extension of $Q$, such that $\bar Q$ restricted to $S^2$ minus the interior of $V$ will have image in $X-St^m(\ast)$ (so that $\alpha$ is homotopically trivial in $X-St^m(\ast)$ - finishing the proof of our theorem). In order to construct $B^3$ and $\bar Q$, we must have homotopies that move 2-cells of $V$ outside of $St^m(\ast)$. Commutation relations will make such homotopies  easy to construct, but if two 2-cells agree on an edge, our two homotopies may take this edge to different destinations. Three dimensional ``product" homotopies are then constructed to compensate for this issue. These homotopies resolve all issues for edges of $V$. Finally we must deal with vertices of $V$. Folding certain 2-cells together and more 3-dimensional  ``product" homotopies will be constructed to resolve those issues. 

\section{Proof of the Main Theorem}\label{PMT}
Before we begin the proof of the main theorem we settle the issue of the number of ends of $G$. As noted earlier, if some some $s\in S$ of our presentation $\langle S,R_1\rangle$ does not appear in any element of $R_1$, then $G$ splits as $\mathbb Z_s\ast \langle S-\{s\}\rangle$ and $G$ is infinite ended. Otherwise, we show $G$ is 1-ended.
\begin{theorem} [Theorem 3, \cite{M4}] \label{combE}  
Suppose $G$ is a finitely generated group, $A$ and $B$ are finitely generated 1 or 2-ended subgroups of $G$, $A\cup B$ generates $G$ and $A\cap B$ is infinite. Then $G$ is 1 or 2-ended.
\end{theorem}
Suppose for each $s\in S$, there is $r_s\in R_1$ with $s$ a letter of $r_s$. Assign the element $t_r\in T$ to each  $r\in R_1$ so that $t_r$ commutes with each letter of $r$.
The abelian group $\langle s,t_{r_s}\rangle$ is 1 or 2-ended and $\langle T\rangle$ is 1-ended. Theorem \ref{combE} implies $\langle T\cup \{s\}\rangle$ is 1-ended. 
Inductively $G$ is 1-ended. 

Now we prove Theorem \ref{Main}. As in \S\ref{OutL}:

\medskip 

\noindent {\it We assume $\alpha$ is an embedded edge path loop with image in $X-St^{M_2(M_1(m)+W)}(\ast)$ with each edge label in $S$ and $V$ is a van Kampen diagram for $\alpha$ with each 2-cell bounded by an edge path labeled by an element of $R_1$. }

\medskip

Also $Q:V\to X$ is the obvious map agreeing with $\alpha$ on $\partial V$.  Our goal is to prove that $\alpha$ is homotopically trivial in $X-St^m(\ast)$.

\medskip

\noindent $(\dagger)$ {\it Say $C$ is a 2-cell of $V$ and $x\in T$ has been assigned to the relation corresponding to $C$.
If $\partial C\cap \partial V\ne\emptyset$.  Choose $v(=v_C)\in \partial C\cap \partial V$.
If $(Q(v)\cdot r_x)\cap St^{M_1(m)+W}(\ast)\ne\emptyset$ then assign $x^{-1}$ to $C$.  Otherwise assign $x$ to $C$.
If $\partial C\cap \partial V=\emptyset$ then assign $x$ to $C$.} 

\medskip

If $C$ is a 2-cell of $V$ containing vertices $v$ and $w$, and $x$ is assigned to $C$, then (see Remark \ref{R1}) $(Q(v)\cdot r_{x})\subset St^W(Q(w)\cdot r_{x})$ (the rays $Q(v)\cdot r_{x}$ and $Q(w)\cdot r_{x}$ are parallel since $x$ commutes with all letters of $C$). 

Now we begin to build $B^3$ by adding 3-cells to $V$, one for each 2-cell of $V$ and we extend $Q$ to these cells. If $C$ is a 2-cell of $V$ and $x\in T$ is assigned to $C$. Consider $Q(C)$ in $X$. Consider  the ``product" homeomorphism (onto its image) $H_C:C\times [0,\infty)\to X$ such that $H_C(c,i)=x^i\cdot Q(c)$ for all $c\in C$ and all integers $i\geq 0$. If $v$ is a vertex of $C$ in $V$ 
then $H_C|_{\{v\}\times [0,\infty)}=Q(v)\cdot r_x(=v\cdot r_x)$. 

\medskip

\noindent ($\dagger\dagger$) {\it If $C$ is a 2-cell of $V$ and $x$ is assigned to $C$ and $v$ is any vertex of $C$, then $H_C(C\times [0,\infty))\subset St^{W}(Q(v)\cdot r_x)$.} 

\medskip

For $C$ a 2-cell of $V$, choose an integer $N_C$ such that $H_C|_{C\times [N_C,\infty)}$ avoids $St^m(\ast)$. Let $N_1$ be the largest of the $N_C$ for $C$ a 2-cell of $V$. Choose $N_2$ greater than $T(v,m)$ (see Lemma \ref{L5}) for all vertices $v$ of $Q(V)$. Choose $N=max\{N_1, N_2\}$. For each 2-cell $C$ of $V$ attach the product 3-cell $C\times [0,N]$ to $V$ by identifying $c\in C$ with $(c,0)\in C\times [0,N]$. If $e$ (respectively $v$) is a common edge (vertex) for cells $C_1$ and $C_2$ of $V$, and the elements of $T^{\pm 1}$ assigned to $C_1$ and $C_2$ are the same, then identify the copy of $e\times [0,N]$ (respectively $\{v\}\times [0,N]$) in $C_1\times [0,N]$  in our attachment space with the copy of $e\times [0,N]$ (respectively $\{v\}\times [0,N]$) in $C_2\times [0,N]$ in the obvious way. Call the resulting space $B_2$. The space $B_2$ will be a subspace of $B^3$. We think of $C\times \{0,N\}$ as being ``shaded". Both of these 2-disks will be part of the boundary of $B^3$. Also if $e$ is an edge in $\partial V$, then $e\times [0,N]$ will be part of the boundary of $B^3$.  If no vertex of $C$ belongs to $\partial V$, then all points of an attached $(C\times (0,N))-(C\times \{0,N\})$ will be interior points of $B^3$.  Extend $Q$ to $B_2$ by $Q(p)=H_{C_i}(p)$ if $p\in C_i\times [0,N]$. This map is well defined and this completes the part of our construction of $B^3$ connected to 2-cells of $V$. Observe that if $C$ is a 2-cell of $V$ then for $p\in C-\partial C$, each point of $\{p\}\times [0,N]$ in $B_2$ satisfies the conditions for a 3-manifold with boundary in $B_2$. Some edges and vertices of $V$ require more attention.  

\begin{lemma}\label{L6}
If $C$ is a 2-cell of $V$ such that  $\partial C\cap \partial V\ne \emptyset$ then $H_C(=H|_{C\times [0,N]})$ has image in $X-St^m(\ast)$. 
\end{lemma}
\begin{proof}
Let $v=v_C\in \partial C\cap \partial V$ and say $x\in T^{\pm1}$ is assigned to $C$.
Suppose 
$$(Q(v)\cdot l_{x})\cap St^{m+W}(\ast)= \emptyset.$$
Statement ($\dagger\dagger$) implies that  $Im(H_C)\subset St^{W}(Q(v)\cdot l_{x})$ (which does not intersect $St^m(\ast)$).
Next suppose 
$$(Q(v)\cdot l_{x})\cap St^{m+W}(\ast)\ne \emptyset .$$ 
Since $M_1(m)\geq m+W$ we have $(Q(v)\cdot l_{x})\cap St^{M_1(m))}(\ast)\ne \emptyset$.
Since $x$ is assigned to $C$, 
 $(Q(v)\cdot r_{x^{-1}})\cap St^{M_1(m)}(\ast)\ne \emptyset$ (see $(\dagger)$). Choose $y\in T-\{x, x^{-1}\}$.
Since $Q(v)\in X-St^{M_1(M_1(m))}(\ast)$, Lemma \ref{L2} (with $x^+$ playing the role of $y^-$, $y$ playing the role of $x$ and $M_1(m)$ playing the role of $m$) implies that 
$$(Q(v)\cdot r_x)\cap St^{M_1(m)}(\ast)=\emptyset.$$ 
Since $St^{m+W}(\ast) \subset St^{M_1(m)}(\ast)$, we have $(Q(v)\cdot r_{x})\cap St^{m+W}(\ast)=\emptyset$. Since $Im(H_C)\subset St^{W}(Q(v)\cdot r_{x})$ we are finished.
\end{proof}

Suppose $e$ is an edge of distinct cells $C_1$ and $C_2$ of $V$. If the same element of $T^{\pm 1}$ has been assigned to $C_1$ and $C_2$, then for $p$ in the interior of $e$, each point of $\{p\}\times [0,N]$ is a 3-manifold point of $B_2$. Consider the edge $e$ to be a map $e:[0,1]\to V$. Suppose $C_1$ is assigned $x$ and $C_2$ is assigned $y$ for $x\ne y^{\pm 1}$. We attach the integer lattice cube complex $K_e= [0,N]\times [0,N]\times [0,1]$ to $B_2$ along $([0,N]\times \{0\}\times [0,1])\cup  (\{0\}\times [0,N]\times [0,1])$ (the two unshaded rectangles containing $e$ in Figure \ref{Fig2}). (Note that our coordinates in Figure \ref{Fig2} do not follow the ``right hand rule".) Our attaching map identifies $(n,0,t)\in [0,N]\times \{0\}\times [0,1]$ with $(e(t),n)\in C_1\times [0,N]$ for all $(t,n)\in [0,1]\times [0,N]$ and $(0,n,t)\in \{0\}\times [0,N]\times [0,1]$ with $(e(t),n)\in C_2\times [0,N]$ for all $(t,n)\in [0,1]\times [0,N]$. 
Next we extend $Q$ (from $B_2\to X$) to $K_e=[0,N]\times[0,N]\times [0,1]$ (as attached to $B_2$). Let the initial vertex of $e$ be $v$ and the terminal vertex be $w$. If $(a,b,c)$ is a vertex of $[0,N]\times[0,N]\times [0,1]$ then define $Q(a,b,0)=Q(v)x^ay^b$ (so if one considers the edge path beginning at $Q(v)$, having the first $a$ edges labeled $x$ and the next $b$ edges labeled $y$, then the end point of this path is $Q(a,b,0)$. Define $Q(a,b,1)=Q(w)x^ay^b$. Extend $Q$ to the remaining parts of  $K_e$ in the obvious way. Our next lemma shows that $Q|_{K_v}$ on the shaded region of Figure \ref{Fig2}, avoids $St^m(\ast)$. 

\begin{figure}
\vbox to 3in{\vspace {-1in} \hspace {-.1in}
\hspace{-1 in}
\includegraphics[scale=1]{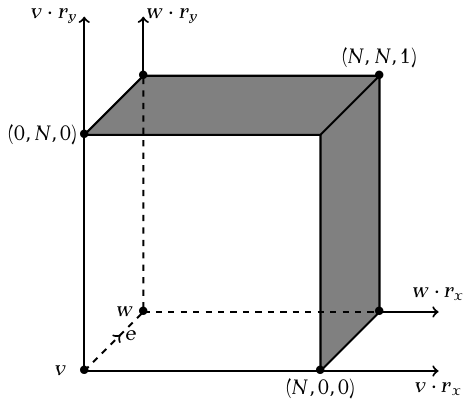}
\vss }
\vspace{.2in}
\caption{Attaching and Mapping $Q|_{K_e}$} 
\label{Fig2}
\vspace{-.1in}
\end{figure}

\begin{lemma}\label{L7}
Suppose $x$ and $y$ are assigned respectively, to 2-cells $C_1$ and $C_2$ of $V$ and $x\ne y^{\pm1}$. If $e=[v,w]$ is a common edge of $C_1$ and $C_2$,  then for $i\in \{0,1\}$: 
$$Q|_{K_e}(([0,N]\times \{N\}\times \{i\})\cup(\{N\}\times [0,N]\times \{i\}))\subset X-St^m(\ast).$$
\end{lemma}
\begin{proof}
Observe that  
$$Im[Q|_{K_v}([0,N]\times \{N\}\times \{0\}]\subset Q(v)\cdot r_y(N)\cdot l_x$$ 
$$Im[Q|_{K_e}([0,N]\times \{N\}\times \{1\}] \subset Q(w)\cdot r_y(N)\cdot l_x$$ 
 $$Im[Q|_{K_e}(\{N\}\times [0,N]\times \{0\}] \subset Q(v)\cdot r_x(N)\cdot l_y \hbox{ and} $$
 $$Im[Q|_{K_e}(\{N\}\times [0,N]\times \{1\}]\subset Q(w)\cdot r_x(N)\cdot l_y.$$ 
 These line segments bound the two shaded rectangles of Figure \ref{Fig2}. Since $N\geq N_2\geq max\{T(v,m), T(w,m)\}$, Lemma \ref{L5} implies that none of these sets intersect $St^m(\ast)$. 
\end{proof}
In the next lemma we retain the hypotheses of Lemma \ref{L7}.
\begin{lemma}\label{L8}
Suppose $x$ and $y$ are assigned respectively, to 2-cells $C_1$ and $C_2$ of $V$,  $x\ne y^{\pm1}$ and  $e=[v,w]$ is a common edge of $C_1$ and $C_2$. Suppose further that $z\in T-\{x^{\pm 1},y^{\pm 1}\}$, $v\in \partial V$ and one of the  three lines $Q(v)\cdot l_x, Q(v)\cdot l_y, Q(v)\cdot l_z$ intersects $St^{M_1(m)}(\ast)$ non-trivially then
$$Im[Q|_{K_e}([0,N]\times [0,N]\times\{0\})]\subset X-St^{M_1(m)}(\ast)(\subset X-St^m(\ast)).$$ 
(This means that the function $Q|_{K_e}$ on the front face of the cube of Figure \ref{Fig2} avoids $St^m(\ast)$. This face will be part of $S^2$, the boundary of $B^3$). 
\end{lemma}
\begin{proof} Let $u_i=v_{C_i}$ for $i\in\{1,2\}$. Suppose that $(Q(v)\cdot l_x)\cap St^{M_1(m)}(\ast)\ne\emptyset$. 
Then $(Q(u_1)\cdot l_x)\cap St^{M_1(m)+W}(\ast)\ne \emptyset$. Since $x$ was assigned to $C_1$, ($\dagger$) implies $(Q(u_1)\cdot r_{x^{-1}})\cap St^{M_1(m)+W}(\ast)\ne \emptyset$. Since $Q(v)\cdot r_{x^{-1}}$ and $Q(u_1)\cdot r_{x^{-1}}$ are parallel: 
$$(Q(v)\cdot r_{x^{-1}})\cap St^{M_1(m)+2W}(\ast)\ne \emptyset.$$ 
Since $v\in \partial V$, 
$$Q(v)\in \alpha\subset X-St^{M_2(M_1(m)+W)}(\ast)\subset X-St^{M_1(M_1(m)+2W)}(\ast).$$ 
Lemma \ref{L2} (with $(M_1(m)+2W$ playing the role of $m$, $x$ playing the role of $y^{-1}$ and $y$ playing the role of $x$) implies that 
$$Q(v)\cdot P(x^+,y)\cap St^{M_1(m)+2W}(\ast)=\emptyset.$$ 
Since  $Im[Q|_{K_e}([0,N]\times [0,N]\times\{0\})]\subset Q(v)\cdot P(x^+,y)$ (see Figure \ref{Fig2}) the case when $(Q(v)\cdot l_x)\cap St^{M_1(m)}(\ast)\ne\emptyset$ is complete. 

If $(Q(v)\cdot l_y)\cap St^{M_1}(m)(\ast)\ne \emptyset$ the argument is the same (interchange  $C_1$, $u_1$, $x$ and $y$ by $C_2$, $u_2$, $y$ and $x$ respectively). 

Next assume $(Q(v)\cdot r_z)\cap St^{M_1(m)}(\ast)\ne \emptyset$. Recall, $v\in X-St^{M_2(M_1(m))}(\ast)$. Lemma \ref{L3} implies that $(Q(v)\cdot R(z^-,x,y)\cap St^{M_1(m)}(\ast)=\emptyset$. Since the vertex set of $Im[Q|_{K_e}([0,N]\times [0,N]\times\{0\})]$ is a subset of $Q(v)\cdot \langle x,y\rangle\subset 
(Q(v)\cdot R(z^-,x,y)$, this case is complete. 
\end{proof}

Once again we assume $x$ and $y$ are assigned respectively, to 2-cells $C_1$ and $C_2$ of $V$,  $x\ne y^{\pm1}$ and  $e=[v,w]$ is a common edge of $C_1$ and $C_2$ (as in Lemma \ref{L7}) and as in Lemma \ref{L8}, assuming that $z\in T-\{x^{\pm1},y^{\pm 1}\}$ and $v\in \partial V$. Suppose that $(Q(v)\cdot l_t)\cap St^{M_1(m)}(\ast)=\emptyset$ for all $t\in T$. If $Q|_{K_e}([0,N]\times [0,N]\times \{0\})\cap St^m(\ast)=\emptyset$ then this face of $K_e$ will be a ``satisfactory" part of the boundary of $B^3$ (as in the conclusion of Lemma \ref{L8}). Otherwise, we make an adjustment.  Assume $p\in Q|_{K_e}([0,N]\times [0,N]\times \{0\})\cap St^m(\ast)$. Let $D(=[0,N]\times [0,N]$) be the domain of $Q|_{K_e}:([0,N]\times [0,N]\times\{0\})\to X$ (in $B_2$).  So $D$ is the shaded region of Figure \ref{Fig3}. 
We will construct a cube $A$ and attach $A$ to $B_2$ along  the 2-disk $D$. We extend the map $Q$ to the cube $A$ so that the boundary of $A$ minus the domain  $D$ has image in $X-St^m(\ast)$. In this way, we replace $D$ by another disk that is mapped into $X-St^m(\ast)$. This will nearly complete our construction of $B^3$ connected to edges of $V$. 

\begin{figure}
\vbox to 3in{\vspace {-1.5in} \hspace {-.1in}
\hspace{-1 in}
\includegraphics[scale=1]{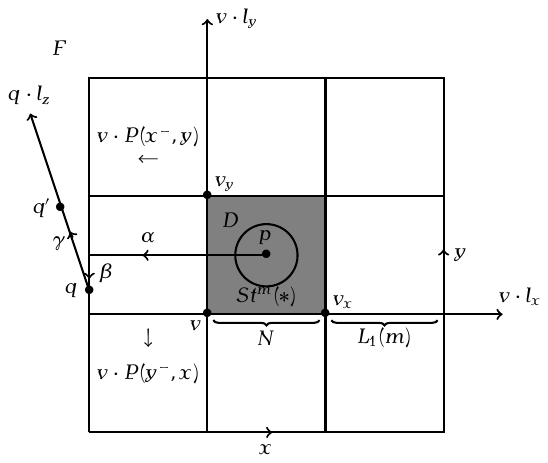}
\vss }
\vspace{.2in}
\caption{Adjusting Faces of 3-Cells} 
\label{Fig3}
\vspace{-.1in}
\end{figure}

Figure \ref{Fig3} depicts the face $F$ of $A$ containing $D$. Here $v_x=Q(v)\cdot x^N(=Q(v)\cdot r_x(N))$ and $v_y=Q(v)\cdot y^N(=Q(v)\cdot r_y(N))$. Since $St^m(\ast)\cap Q(D)\ne\emptyset$,  $Q(v)P(x^+,y)\cap St^m(\ast)\ne \emptyset \ne Q(v)P(y^+,x)\cap St^m(\ast)$ (see Figure \ref{Fig3}). 
Combining with the fact that  $Q(v)\cdot l_x\cap St^{M_1(m)}(\ast)=\emptyset=Q_v\cdot l_y\cap St^{M_1(m)}(\ast)$, Lemma \ref{L1} implies that both the sets $Q(v)\cdot P(x^-, y)$, and $Q(v)\cdot P(y^-, x)$, have trivial intersection with $St^m(\ast)$.

Since $N\geq T(v,m)$, Lemma \ref{L5} implies that $v_y\cdot P(y^+,x)$ and $v_x\cdot P(x^+,y)$ have trivial intersection with $St^m(\ast)$.
Hence the $(2L_1(m)+N)$ by  $(2L_1(m)+N)$ rectangle of Figure \ref{Fig3} is such that the part outside of $D$ (unshaded) avoids $St^m(v)$. If $q$ is a point on the boundary of $F$ (see Figure \ref{Fig3}) then an $(x,y,z)$ geodesic from $p\in St^m(\ast)$ to a point $q'$ of $q\cdot l_z$ has the form $(\alpha, \beta, \gamma)$ where each edge of $\alpha$ has label $x$ (or $x^{-1}$), each each label of $\beta$ has label $y$ (or $y^{-1}$) and each label of $\gamma$ has label $z$ (or $z^{-1}$). Either $\alpha$ or $\beta$ has length $\geq L_1(m)$. Hence the definition of $L_1(m)$ implies that  $(q\cdot l_z)\cap St^m(\ast)=\emptyset$ for all boundary points $q$ of $F$.
To finish our construction of $A$, simply choose an integer $J$, so that the $z^J$ translate of $Q(F)$ avoids $St^m(\ast)$.  Map the integer lattice cube $F\times [0,J]$ to $X$ in the obvious way. Note that  $\partial (F\times [0,J])-D$ maps to $X-St(\ast)$ (this set will the part of $S^3=\partial B^3$ replacing $D$).

One of the hypotheses of Lemma \ref{L7} is that $x\ne y^{\pm 1}$. If $x=x^{\pm 1}$ then choose $z\in  T-\{x^{\pm 1}\}$ such that $z$ commutes with the label of $e$ (use part (2) of the hypotheses for Theorem \ref{Main}). Form two cubes of the form $[0,N]\times [0,N]\times [0,1]$. Attach as in Figure \ref{Fig4} and map into $X$ as in the previous cases (attaching an additional disk on a front face if necessary).

One more issue must be addressed before we move from considering edges to considering vertices of $V$. It may be that cells $C_1$ and $C_2$ of $V$ share  consecutive edges. If $C_1$ and $C_2$ share $k$-consecutive edges, then instead of attaching a copy of $[0,N]\times [0,N]\times [0,1]$ attach a copy of $[0,N]\times [0,N]\times [0,k]$ in the obvious way.

\begin{figure}
\vbox to 3in{\vspace {-2in} \hspace {-.9in}
\hspace{-1 in}
\includegraphics[scale=1]{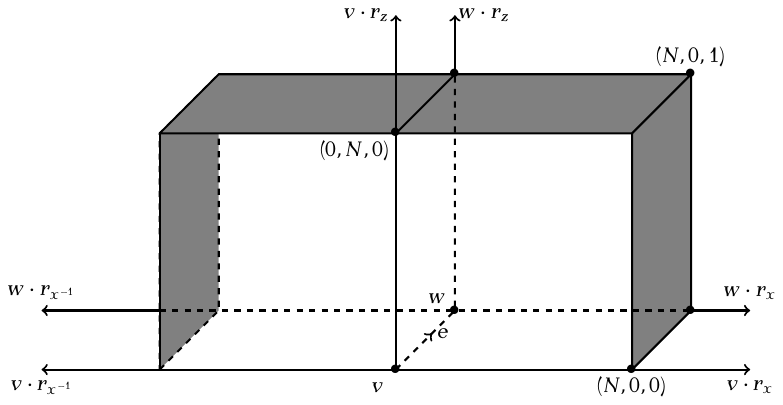}
\vss }
\vspace{-.8in}
\caption{Double Attaching and Mapping} 
\label{Fig4}
\vspace{-.1in}
\end{figure}

Let $B_1$ be $B_2$ union all of the 3-cells corresponding to the $K_e$ along with the 3-cells attached to faces $D$ of this last type. Note that each 3-cell that we have attached so far strong deformation retracts to its attaching set and so $B_1$ is contractible. When retracting a 3-cell of the type $[0,N]\times [0,N]\times [0,1]$ (see Figure \ref{Fig2}) to its attaching set, it is important to define a single deformation retract defined on $[0,N]\times [0,N]\times [0,1]$ and use this deformation retract to define all deformation retracts in $B_1$ of 3-cells of this type. The deformation we choose is a composition of two strong deformation retractions. The first is a strong deformation retraction of $[0,N]\times [0,N]\times [0,1]$ to the four unshaded faces of Figure \ref{Fig2} ($([0,N]\times[0,N]\times \{0,1\})\cup ([0,N]\times \{0\}\times [0,1])\cup (\{0\}\times [0,N]\times[0,1]$). The second strong deformation retraction deforms the front and back face ($([0,N]\times[0,N]\times \{0,1\})$) of Figure \ref{Fig2}, to the line segments  $([0,N]\times \{0\}\times \{0\})\cup (\{0\}\times [0,N]\times \{0\})$ (respectively $([0,N]\times \{0\}\times \{1\})\cup (\{0\}\times [0,N]\times \{1\})$). 
In the next step we may fold the free face $[0,N]\times [0,N]\times \{0\}$) of a 3-cell to another free face of an adjacent 3-cell. In order that these two 3-cells strong deformation retract to $B_2$ we will apply the first type of deformation to these two copies of $[0,N]\times [0,N]\times[0,1]$ and then later apply the second type along the now common front face and the back faces.   
The shaded regions of Figure \ref{Fig2} (or Figure \ref{Fig4}) will be part of the boundary of $B^3$. Each point of $[0,N]\times [0,N]\times (0,1)$ in $B_1$ is a 3-manifold point of $B_1$. At this point we are finished attaching 3-cells corresponding to 1-cells of $V$. 

Finally we must attach integer lattice cube complexes of the form $[0,N]\times [0,N]\times [0,N]$ 
for certain vertices $v$ of $V$. This will be done so that $(0,0,0)$ will be attached to $v$. If $v$ is a boundary vertex of $V$ then there is nothing to do. If $v$ is an interior vertex of $V$ then let $C_1,\ldots ,C_n$ be the 2-cells of $V$ containing $v$ such that $C_i$ and $C_{i+1}$ share the edge $e_i=[v,w_i]$ and $C_1$ and $C_n$ share the edge $e_n=[v,w_n]$. Say $x_i\in T^{\pm 1}$ is assigned to $C_i$. If $x_1=x_2$ we did not attach a copy of $[0,N]\times [0,N]\times [0,1]$ (or two copies of $[0,N]\times [0,N]\times [0,1]$) to $e_1$ in the previous step and there is no action taken for $e_1$. If all $x_i$ are the same, then $v$ is already a manifold point. Say $e_i$ is an edge such that $x_i\ne x_{i+1}$. If $x_i\ne x_{i+1}^{\pm 1}$ then there is a free face $\bar F_i=[0,N]_i\times[0,N]_i\times \{0\}\subset B_1$ such that $Q(0,0,0)=v$, $Q(t,0,0)=v\cdot r_{x_i}(t)$ and $Q(0,t,0)=r_{x_{i+1}}(t)$ for $t\in [0,N]_i$. If $x_i=x_{i+1}^{-1}$ there are two free faces at $v$ and an additional $v\cdot r_z$ (see Figure \ref{Fig4}) to consider. In what follows, we will either attach copies of $[0,N]^3$ to these faces or ``fold" two such adjacent faces to complete our construction. 

Let $x_{j_1}=x_1$. Inductively, let  $j_i$ be the first integer following $j_{i-1}$ such that $r_{x_{j_i}}\ne r_{x_{j_{i-1}}}$. For simplicity, let $y_i=x_{j_i}$. Our ``picture" in $B_1$ (see Figure \ref{Fig5}) is a collection of free faces $F_i(=\bar F_{j_i})$ such that $Q(0,0,0)=v$, $Q(t,0,0)=v\cdot r_{y_i}(t)$ and $Q(0,t,0)=r_{y_{i+1}}(t)$ for $t\in [0,N]_i$. The point $v$ is the only point of the face $F_i$ in $V$. 
Here $y_i\ne y_{i+1}$ and $y_i\ne y_i^{-1}$ (since we consider both faces of Figure \ref{Fig4}, along with $v\cdot r_z$).  Suppose $y_1=y_3$. We fold the two faces $F_1$ and $F_2$ together in $B_1$ by identifying $(a,b,0)\in F_1$ with $(b,a,0)\in F_3$ (since $Q$ agrees on these two points). The interior points of $F_1$ (and $F_2$) become interior 3-manifold points of $B^3$. 
\begin{figure}
\vbox to 3in{\vspace {-2in} \hspace {.1in}
\hspace{-1 in}
\includegraphics[scale=1]{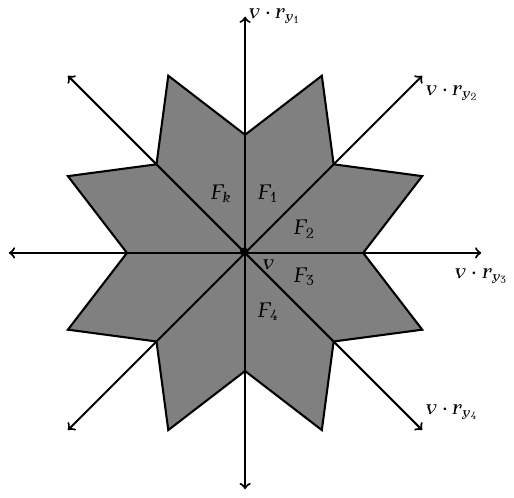}
\vss }
\vspace{-.1in}
\caption{Adjusting vertex neighborhoods} 
\label{Fig5}
\vspace{-.1in}
\end{figure}
If $F_1$ and $F_2$ are identified, then our new ``picture" is the same as Figure \ref{Fig5}, with $F_1$ and $F_2$ removed. (So we cycle around $v$ by beginning with face $F_3$ and ending with face $ F_k$.) This means we may assume not only that $y_i\ne y_{i+1}^{\pm1}$ but also that $y_i\ne y_{i+2}$. Next assume that $\{y_1,y_2,y_3\}$ generates a copy of $\mathbb Z^3$. Then we attach a copy of $A=[0,N]^3$ to $F_1\cup F_2$ so that $(0,0,0)$ is attached to $v$, $[0,N]\times [0,N]\times \{0\}$ is attached to $F_1$, $[0,N]\times \{0\}\times [0,N]$ is attached to $F_2$. Map this cube into $X$ in the obvious way, respecting the product structure. The three faces containing $(N,N,N)$ are mapped to $X-St^m(\ast)$ by the definition of $N$ (and will be part of the boundary of $B$). The face $\{0\}\times [0,N]\times [0,N]$ will now be a free face replacing $F_1$ and $F_2$. This face will have a side of length $N$ on each of $v\cdot r_{y_1}$ and $v\cdot r_{y_3}$. This sort of reduction allows us to assume (perhaps after relabeling) that no three (cyclically) consecutive $y_i$ generate a copy of $\mathbb Z^3$.  Combining this with the inequalities $y_i\ne y_{i+1}^{\pm1}$ and $y_{i+2}\ne y_i$ forces $y_{i+2}=y_i^{-1}$. In particular, the set of all remaining $y_i$ generate a copy of $\mathbb Z^2$. Choose $z\in T$ such that $z$ is not in the group generated by the $y_i$. Consider the shaded 2-disk $D=\cup_{i=1}^kF_i$ of Figure \ref{Fig5}. Let $E=D\times [0,N]$. Attach $E$ to $D$ by identifying $d$ with $(d,0)$ for all  $d\in D$. Extend $Q$ to map $E$ into $X$ respecting the product structure. 

The resulting space $B^3$ is a compact 3-manifold with boundary. To see that $B^3$ is contractible, first undo the last step by a strong deformation retraction of $E$ to $D$. For each of the copies of $[0,N]^3$, strong deformation retract these to the three faces containing $(0,0,0)$ (the other three faces are boundary faces of $B^3$). For the 3-cells of the form $[0,N]\times[0,N\times [0,1]$, use the first type of strong deformation retraction (to the four unshaded faces of Figures \ref{Fig2} or \ref{Fig4}). Next strong deformation retract copies of $F\times [0,J]$ (see Figure 3) to $D$. If a 2-dimensional face remains from the  $[0,N]\times[0,N\times [0,1]$
deformations of the first type then use the strong deformation retractions of the second type to deform these faces. Finally, for each 2-cell $C$ of $V$, strong deformation retract $C\times [0,N]$ to $C$. The result is $V$. The 3-manifold $B^3$ is contractible and hence a 3-ball. The set $\partial B-int(V)$ maps to $X-St^m(\ast)$ (by the construction). Hence $\alpha$ (a map of the boundary of $V$ to $X-St^{M_2(M_1(m)+W)}(\ast)$) is homotopically trivial in $X-St^m(\ast)$ and $X$ (and hence $G$) is simply connected at infinity. 

 \section{Presentations for Mapping Class Groups}\label{PMCG}
In order to see that our presentations of $F_{3,0}^0$ and $F_{4,0}^0$ satisfy the hypotheses of our theorem first consider $F_{3,0}^0$. We use a presentation of $F_{3,0}^0$ given in \cite{Ge01}. A generating set for this presentation is 
$$S=\{a_1,a_2,a_3,a_4, b,b_1,b_2, c_{i,j} \hbox{ for } 1\leq i,j\leq 4\ \ \  i\ne j\}.$$
We begin by listing non-commuting pairs.  The notation $(x:y)$ means that $x$ does not commute with $y$ and we define $(x:y_1,\ldots, y_n)=(x,y_1),\ldots, (x,y_n)$ so that $x$ does not commute with $y_i$ for all $i$.  Some caution here: we list $(b_1:a_2)$ but there are many more generators that $b_1$ does not commute with that occur later in our list. 

$(b:a_1,a_2,a_3,a_4)$,  \ \ $(b_1:a_2) \ \ (b_2:a_4),$ 

$(c_{1,2}, b_1),$\ \ $(c_{1,3}: a_2,c_{2,4}, c_{2,1})$,\ \ $(c_{1,4}:a_2,a_3,b_2,c_{2,1},c_{3,1},c_{3,2}, c_{4,2}, c_{4,3}),$

$(c_{2,1}: b_1, a_3,a_4, c_{3,2}, c_{4,2}),$\ \ $(c_{2,3}: b_1),$\ \ $(c_{2,4}:b_1,b_2,a_3, c_{3,1}, c_{3,2}, c_{4,3}),$

$(c_{3,1}: a_4, c_{4,2}, c_{4,3}), $\ \ $(c_{3,2}:a_4,a_1, b_1, c_{4,3}),$\ \ $(c_{3,4}: b_2),$

$(c_{4,1}, b_2),$\ \ $(c_{4,2}:b_2, a_1,b_1),$ \ \ $(c_{4,3}: b_2,a_1,a_2)$

There are three types of relations. The {\it commutation relations} of the form $xy=yx$ occur if $(x:y)$ does not appear above. The {\it Artin} relations $xyx=yxy$ holds if $(x:y)$ appears above. 
Finally there are {\it star} relations $c_{i,j}c_{j,k}c_{k,i}=(a_ia_ja_kb)^3$ for certain ``good" $(i,j,k)$. Here the term $c_{1,1}$ is allowed to occur and is defined as equal to 1. 
Let $R$ be the subset of the free group $F(S)$ consisting of all of these words. We are somewhat vague on the star relations, since our analysis does not rely on the specifics of these relations. In \cite{Ge01} a ``good" triple $(i,j,k)$ is defined and only star relations such that $(i,j,k)$ is good are actually used here. 
 An important aspect to this presentation is that $c_{1,2}$ (respectively $c_{3,4}$) commutes with all generators except $b_1$ (respectively $b_2$) and $c_{1,3}$ commutes with all generators except $a_2$, $c_{2,4}$ and $c_{2,1}$. Notice that all three of the generators $c_{1,2}$, $c_{3,4}$ and $c_{1,3}$ commute with one another. We show that the set $T=\{c_{1,2},c_{3,4},c_{1,3}\}$  satisfies the conditions of Theorem \ref{Main}.    
To each star relation we assign $c_{1,2}$. The element $c_{3,4}$ also commutes with all letters of each star relation and so conditions (1) and (2) are satisfied for all star relations.  Any other relation only involves 2-generators. We assign $c_{1,3}$ to $b_1b_2=b_2b_1$ to satisfy (1). If in condition (2) $s=b_1$ then let $z=c_{3,4}$ if $s=b_2$, let $z=c_{1,2}$. If a remaining relation $r$ does not contain $b_1$ then we assign $c_{1,2}$ to that relation. If in condition (2), $s=b_2$ let $z=c_{1,3}$. If $s\ne b_2$ then let $z=c_{3,4}$. Finally, if $b_1$ is a letter of $r(\ne b_1b_2b_1^{-1}b_2^{-1})$, then we assigned $c_{3,4}$ to $r$. For condition (2), if $s=b_1$ let $z=c_{1,3}$. If $s\ne b_1$ let $z=c_{1,2}$. 

The elements $c_{1,2}$, $c_{3,4}$ and $c_{1,3}$  appear in the \cite{Ge01}  presentation of each mapping class group of a closed surface of genus $\geq 3$ with the same non-commutation properties and so Theorem \ref{Main} applies to these groups as well.

\bibliographystyle{amsalpha}
\bibliography{paper1}{}

\end{document}